\documentclass[10pt,leqno]{amsart}
\usepackage{amsmath,amsfonts,amssymb,amscd,amsthm,amsbsy,upref}
\usepackage{upgreek}
\usepackage{color}

\newtheorem{thm}{Theorem}%[section]
\newtheorem{lem}[thm]{Lemma}
\newtheorem{clm}[thm]{Claim}
\newtheorem{rem}[thm]{Remark}
\newtheorem{defn}[thm]{Definition}

\newtheorem{prop}[thm]{Proposition}

\newtheorem{theorem}{Theorem}                                                                                %
\newcommand{\h}[1]{\textbf{#1}} % \bf-s the letters in math mode
\newcommand\R{\mathbb R}

\newcommand\dist{\text{dist}}
\newcommand\na{{\nabla}}
\newcommand\wt{\widetilde}
\usepackage{esint}

\theoremstyle{definition}
\newcommand\fb[1]{\p \{ {#1} > 0 \}} % free boundary

\def\ds{\displaystyle}

\def\ds{\displaystyle}

\def\C{\mathcal C} %% Used in the cylinder
\def\cal{\mathcal}
\def\p{\partial}
\def\Om{\Omega}
\def\D{\nabla}
\def\e{\varepsilon}
\def\div{\operatorname{div}}
\newcommand\ol{\overline}
\usepackage{mathrsfs}
\newcommand{\rred}[1]{\textcolor{black}{#1}}
\newcommand{\bblue}[1]{\textcolor{black}{#1}}

\begin{document}

\title[Quasilinear continuous casting problem]{Regularity for a quasilinear continuous casting problem }
\author[A.L. Karakhanyan]{Aram L. Karakhanyan}
\address[Aram L. Karakhanyan]{Maxwell Institute for Mathematical Sciences and
School of Mathematics, University of Edinburgh,
James Clerk Maxwell Building, Peter Guthrie Tait Road,
Edinburgh EH9 3FD,
United Kingdom.
}
\email{aram.karakhanyan@ed.ac.uk}

\thanks{ 
Free boundary regularity, two phase, $p-$Laplace, Stefan problem, continuous casting.}
\thanks{The author was supported in part by EPSRC grant EP/K024566/1}
%%%%%%%%%%%%

\begin{abstract}
In this paper we study the regularity of weak solutions to the continuous casting problem 
\begin{equation}
\div(|\nabla u|^{p-2}\na u-\h v\beta(u))=0\tag{$\sharp$}
\end{equation}
 for prescribed 
constant velocity $\h v$ and enthalpy $\beta(u)$ with jump discontinuity at $u=0$. We establish the following estimates: 
local log-Lipschitz $p>2$ for $u$ (and BMO for $\na u$) for two phase, Lipschitz 
$p>1$ for one phase and linear growth up-to boundary near the contact points. 
We also prove that the 
free boundary is continuous curve in the direction of $\h v$ in two 
spatial dimensions. The proof is based on a delicate argument exploiting 
Sard's theorem for $W^{2, 2+\eta}, \eta>0$ functions and circumventing the lack of 
comparison principle for the solutions of ($\sharp$). 
\end{abstract}

\maketitle

%%%%%%%%%%%%
\baselineskip=16pt              %% draft mode 18pt
%%%%%%%%%%%%

%%%%%%%%%%
%%         SECTION
%%%%%%%%%%
\section{Introduction}

Let $\Omega\subset \R^{N-1}$ be a bounded domain with $C^{1, \alpha}$ boundary. 
Denote  $\C_L=\Omega\times (0, L)\subset \R^N$ where $L>0$ is given. In what 
follows we denote the points in $\C_L$ by $X=(x, z)$ where $x\in\Omega, z\in(0,L)$.

Consider the steady state continuous casting problem in 
$\C_L$ with constant convection in the direction of the $z$-axis

\begin{equation}\label{mainpde}
\left\{
\begin{array}{lll}
\Delta_p u=\p_z\beta(u) & \text{in}\ \  \C_L, \\
u=m^+ &\text{on}\ \  \Omega\times\{L\},\\
u=-m^- &\text{on}\ \  \Omega\times\{0\},\\
u=g& \text{on}\ \  \partial \Omega\times (0, L),
\end{array}
\right.
\end{equation}
where $m^+, m^-$ are two positive constants.
The quasilinear degenerate elliptic operator 
$$\Delta_p u:=\div(|\nabla u|^{p-2}\nabla u),\quad  1<p<\infty$$
is called the $p-$Laplacian.
The boundary data $g$ on the lateral boundary of $\C_L$ is $C^{1, \alpha}$ regular for some $\alpha\in (0,1)$ and 
satisfies the compatibility conditions
\begin{equation}\label{2}
g(x,0)=-m^-, \quad g(x,L)=m^+ \quad \forall x\in \p\Omega.
\end{equation}

The enthalpy
$\beta=\beta(u)$ is defined as follows
\begin{eqnarray}\label{beta}
  \beta(s)=\left\{\begin{array}{lll}
                          a s &\textrm{if}  &s<0,\\
                          \in [0,\ell] &\textrm{if } &s=0,\\
                          a s+\ell &\textrm{if} &s>0.
                    \end{array}
                    \right.
\end{eqnarray}
Here $a>0, \ell>0$ are given constants. An equivalent definition of $\beta$, which will useful in the 
analysis of the equation $\Delta_p u=\p_z\beta(u)$, 
is $\beta(s)=as+H(s)$, with
$H(s)$ being the Heaviside function 
\begin{equation*}
H(s)=\left\{
\begin{array}{ll}
s, &\hbox{if}\  s>0, \\
0. &\hbox{if}\  s\le 0.
\end{array}
\right.
\end{equation*}
The equations of the form  
\begin{equation}\label{e-balance}
 \Delta_p u=\div[\h v\beta(u)]+f
\end{equation}
have a number of physical applications \cite{Bear}. 
One may  interpret $u(X)$
as the normalized temperature at a point $X\in \C_L$ whereas
$f$ accounts for sources and $\h v(X)$ is the velocity of convection.
(\ref{e-balance}) manifests the heat conservation of thermodynamical system with enthalpy $\beta(u)$
when the liquid phase has velocity $\h v$. The case of constant convection 
$\h v=e_N$ models the solidification of molten steel extracted at constant speed and is used 
intensively in steel production. We shall mainly focus on this case.

In order to study the problem mathematically we first formulate it in weak sense.
Let  $f$ be a given continuous function and $\h v$ a Lipschitz continuous vectorfield defined in the cylinder
$\C_L$. In what follows $W^{1, p}(\C_L), p>1$ denotes the Sobolev space of weakly differentiable functions $v\in L^{p}$ such that the weak derivatives of $v$ are in $L^p(\C_L)$. The subspace of $W^{1, p}$ of  functions with vanishing trace is denoted by $W^{1, p}_0.$

\smallskip

\begin{defn}\label{def-suka}
Let $\h v\in C^{0, 1}(\ol{\C_L})$ and $f$ continuous. Then $u\in W^{1,p}(\C_L)$
is said to be a weak solution of  (\ref{e-balance}) in $\C_L$ if
\begin{eqnarray}\label{D-id}
-\int_{\C_L}\beta(u)\h v\cdot \na\phi+\int_{\C_L}|\nabla u|^{p-2}\na u\na\phi=-\int_{\C_L}f\phi, \qquad \forall\phi\in C^\infty_0(\C_L).
\end{eqnarray}
 is satisfied.
Here $\beta $  is the maximal monotone graph given by (\ref{beta}).
\end{defn}

For given  function $g\in C^{1,\alpha}(\p\C_L), \alpha\in(0,1)$ we consider the
weak solutions to Dirichlet problem
\begin{align}\label{DP}
\left\{
   \begin{array}{lll}\tag{\textit{\bf DP}}
    \Delta_p u=\div [\h v \beta(u)]+f \hspace{5pt} &{\rm{in}} \ {\C}_L,\\
u(x,0)=-m^-\  &x\in \Om,\\
   u(x,L)=m^+\ & x\in \Om,\\
    u=g(X) & {\rm on } \ \Sigma=\partial\Omega\times (0,L).
   \end{array}
   \right.
\end{align}

\smallskip

\begin{defn}\label{DP-def} Let $\h v\in C^{0, 1}(\ol{\C_L})$.
A pair $(u, \eta)$ is said to be a weak solution to \eqref{DP}
if $u\in W^{1,p}(\C_L), \eta\in \beta(u)$, $u=g$ on $\Sigma:=\partial\Omega\times(0,L)$ (in the trace sense), $u(x,0)=-m^-, u(x,L)=m^+, x\in \Om$
and for any $\phi\in W_0^{1,p}(\C_L)$
\begin{eqnarray}\label{DP-id}
-\int_{\C_L}\eta\h v\cdot \na \phi+\int_{\C_L}|\na u|^{p-2}\na u\na\phi=-\int_{\C_L}f\phi.
\end{eqnarray}
\end{defn}
The condition $\h v\in C^{0, 1}(\ol{\C_L})$ on the convection $\h v$ is of technical nature and later will 
be replaced by a stronger one, namely $\h v= \h e_N$ which corresponds to the continuous casting problem.
\smallskip
%% Free boundary condition
Let $u^+=\max(u, 0), u^-=-\min(u, 0)$ so that $u=u^+-u^-$. If $\fb u$ is $C^1$ smooth then the following free boundary condition is satisfied 
\begin{equation}\label{Stefan-cond}
|\na u^+|^{p-2}\na u^+\cdot \nu^+-|\na u^-|^{p-2}\na u^-\cdot \nu^-=\ell \h v\cdot \nu^+
\end{equation}
where $\nu^+, \nu^-$ are the outer normals of $\{u>0\}\cap \C_L$ and $\{u<0\}\cap \C_L$, respectively, see \cite{Urbano} equation (5).
\begin{rem}\label{Holder}
 It is known that there is a unique weak solution of the problem such that 
$\|u\|_\infty\le M<\infty$, see Theorems 2 and 3 in  \cite{Urbano}. For the 
classical case $p=2$ we refer to   
\cite{Rod-rev} Theorem 4.14 where it is shown that  $u\in C^\alpha(\overline{\C_L})$ provided that $g\in C^\alpha(\partial \C_L)$.
\end{rem}
%%%%%%%%%%
%%         SECTION
%%%%%%%%%%
\section{Main results}

In this section we formulate our main results.
\smallskip
\begin{theorem}\label{thm-bmo}
Let $2<p<\infty$ and $u$ be a bounded weak solution to the 
equation $\Delta_p u=\partial_z(\beta(u))$. Then $u$ is locally in $BMO$ and 
consequently it is locally log-Lipschitz
continuous in $\C_L$.
\end{theorem}

Notice that in Theorem \ref{thm-bmo} the weak solution $u$ may change sign.
The condition $p>2$ is dictated  by the non-variational structure  of this equation. Indeed, 
as we shall see below (see Remark  \ref{non-var})  for $1<p<2$ 
our technique gives only H\"older continuity of $u$. 

\smallskip 

It is worthwhile 
to point out that for one phase problem, $p>2$,  the BMO estimate above implies 
a linear growth from free boundary, see Lemma \ref{Harnack}. 
However, the same conclusion holds for any $p>1$ as the next theorem shows.

\begin{theorem}\label{Lip}
Let $u$ be a non-negative bounded weak solution to (\ref{e-balance}) and $1<p<\infty$.
\begin{itemize}
\item[$\bf 1^\circ$]Then  $u$ \rred{grows linearly away from the free boundary $\partial \{u>0\}\cap \C_L$}, provided that $\h v \in \h L^{\infty}(\C_L, \R^N)$
 and $f\in C(\ol{\C_L})$. 
 This means that for every subdomain $D\subset \subset \C_L$ there is a constant $C$ depending only on 
 $N, p, a, L,\ell,  \dist(\overline D, \p\C_L), \|\h v\|_\infty, \|f\|_{C}$ such that 
 \[u(x)\le C|x-x_0|, \quad x\in D, x_0\in D\cap \fb u.\]
\item[$\bf 2^\circ$] \rred{Furthermore if $\h v=e_N$ then $u$ is locally Lipschitz continuous in $\C_L$.}
\end{itemize}
\end{theorem}

For $p=2$ the local regularity for two phase problem is discussed in \cite{K-2}, and \cite{K-R}.
The regularity of free boundary is more delicate, our main result here
states that if $N=2$ and $u$ is a Lipschitz continuous solution of  (\ref{DP}) and $\p_z u \geq 0$, then
the free boundary
is a  continuous graph in $z-$direction. In order to prove this result we first show that 
for suitable boundary data $g$ we have $\partial_z u\ge 0$.

\begin{prop}\label{cor-1}
Let $u\ge 0$ be a weak solution of \eqref{DP} in the sense of Definition \ref{DP-def}, $N=2<p<\infty, m^-=0$ 
and assume further that
\begin{equation}\label{g-cond}
\liminf_{z\to z_0}\frac{g(x,z)-g(x, z_0)}{z-z_0}\ge 0, \quad \forall x\in \p\Om, z_0\in [0,L],\quad \p_z g(X)=0, \ X\in \p \Sigma,
\end{equation}
where $g\in W^{2, 2+\eta_0}(\Sigma), \eta_0>0, $ and 
$\Sigma$ is the lateral boundary of $\C_L$.
Then $u$ is monotone nondecreasing in $z$ direction.
\end{prop}

%%%%%%%%%

Finally we formulate our main result concerning the regularity of free boundary in two spatial dimensions.
\begin{theorem}\label{FB}
 Let  $u$ be a nonnegative
weak solution to (\ref{DP}) in $\C_L, N=2<p, m^-=0$
such that $u$ is nondecreasing  in $z-$direction. Let $g\in W^{2, 2+\eta_0}(\Sigma), \eta_0>0, $ where 
$\Sigma$ is the lateral boundary of $\C_L$.
Then for any subdomain $D\subset \C_L$, $\Gamma(u)=\partial \{u>0\}\cap D$ is locally a continuous graph in $\h e_2-$direction.
\end{theorem}

The main difficulty in the proof is the lack of the ellipticity of the operator
$\Delta_p$. We circumvent this difficulty by a delicate argument based on an approximation of $u$ 
and Sard's theorem for $W^{2, 2+\eta}$ functions.  

\smallskip
\begin{rem}
Notice that if $g\ge 0$, i.e. we consider the one phase problem, then 
for $g\in C^{1, \alpha}, \alpha>0$ we cannot have strict monotone (i.e. strict inequality in \eqref{g-cond}) boundary condition \eqref{g-cond} because 
at the free boundary points on the lateral boundary $\Sigma=\partial\Omega\times(0, L)$ one has $\p_zg=0$ as $g=0$ is a minimal value.
\end{rem}

%%%%%%%%%%%%
\begin{rem}
One can take more general boundary data and consider the following problem
\begin{align}
\left\{
   \begin{array}{lll}
    \Delta_p u=\div [\h v \beta(u)]+f \hspace{5pt} &{\rm{in}} \ {\C}_L,\\
u(x,0)=h_0(x)\ & x\in \Om,\\
   u(x,L)=h_L(x) \ & x\in \Om,\\
    u=g(X) \hspace{6pt} & {\rm on } \ \Sigma=\partial\Omega\times (0,L).
   \end{array}
   \right.
\end{align}
One can extend all of the results to this general case under suitable conditions on 
$\h v$ and $f$ and the boundary data $h_0, h_1$. For instance if $f=0$ and $(\h v\cdot e_N)\ge 0$ with $\h v\in C^{0, 1}(\C_L)$
then the free boundary is a continuous curve in the $z$ direction in two spatial dimensions.

\end{rem}

The paper is organised as follows: In Section \ref{sec-bmo} we prove some
BMO estimates by testing $u$ against its $p-$harmonic replacement in small balls. 
Theorem \ref{thm-bmo} will follow as a consequence of Lemma \ref{lemma:coherence}.
In Section \ref{sec-thm2} we prove Theorem \ref{Lip}. The argument is based on a 
dyadic scaling method.  The regularity of the free boundary in two spatial dimensions
is discussed in Section \ref{sec-thm3}. In Section \ref{sec-gurev} we prove that 
nonnegative  solutions $u$ have at most linear growth at the contact points where 
the free boundary touches the fixed boundary.   

We shall also sketch how one can extend the results to uniformly elliptic quasilinear equations in Section \ref{gens}.
The paper also contains the proofs of a version of Caccioppoli type estimate and Hopf's lemma included in the Appendix.
%%%%%%%%%%
%%         SECTION
%%%%%%%%%%
\section{Notations}

\begin{tabbing}
$C_0, C_1, C_D \dots$ \hspace{2.65cm}       \=\hbox{generic constants }\\
$\chi_D$            \>\hbox{the characteristic function of a set }$D \subset \mathbb R^N,\  N\geq 2$\\
$\overline \Omega$       \>\hbox{the closure of } $\Omega$\\
$\partial \Omega$        \>\hbox{the  boundary of }  $\Omega$ \\
$\nu$ \> outer unit normal  \\
$X=(x, z)\in \R^N$             \>$x=(x_1,\dots , x_{N-1}, 0) $\\
$\D u$   \> $\D u=(\p_{x_1}u, \p_{x_2}u, \dots, \p_{z}u), \ds
\p_{X_i}=\frac{\p}{\p X_i}, 1\le i\le N-1, \p_z =\frac{\p }{\p z}$\\
$(\na u)_{x_0,\rho}$\> $ (\na u)_{x_0,\rho}:=\fint_{B_\rho(x_0)}\nabla u$\\
$\C_L$         \>\hbox{the cylinder $\C_L=\Omega\times(0, L), L>0$ for some $\Omega\subset \R^{N-1}$}\\
$\Sigma$ \> \hbox{lateral boundary of $C_L$}, \  $\partial\Omega\times(0, L)$\\
$B_r(X)$   \qquad    \>$\{Y\in \R^N: |Y - X|<r\}$\\
$B_r $     \>$B_r(0)$
\end{tabbing}

%%%%%%%%%%
%%         SECTION
%%%%%%%%%%
\section{BMO estimate}\label{sec-bmo}
\begin{lem}[Continuity of weak solutions]\label{lemma:coherence}
Let~$u\in W^{1,p}(\C_L)$ be a solution of~\eqref{mainpde}. Then 
 for $p>2$, 
there exist $c>0$ and $B>0$ depending only on $a, \ell, p, N$ and $\sup_{C_L}|u|$ such that
$$\phi(r)\leq cr^{\frac N2}\left(\frac{\phi(R)}{R^{\frac N2}}+B\right),$$
for all $0<r\le R\le \dist(X_0, \partial\C_L)$,
where $$ \phi(r):=\sup\limits_{t\leq r} \|\na u-(\na u)_{X_0,t}\|_{L^2(B_t(X_0))}.
$$  and $X_0\in \fb u$.

In particular, we have that $\na u \in BMO(D)$, 
for any bounded subdomain $D\Subset\C_L$,
and thus $u$ is locally log-Lipschitz continuous.
Furthermore, $\na u\in L^q(D)$ for any $1<q<\infty$.
\end{lem}

\begin{proof} Fix~$R\ge r>0$ and $x_0\in D$ such that $B_{2R(x_0)}\Subset D$.
Let $v$ be the solution of 
$$ \left\{\begin{array}{ll}
\Delta_p v=0 & {\mbox{ in }} B_{2R}(X_0),\\
v=u & {\mbox{ on }} \partial B_{2R}(X_0). 
\end{array}
\right.$$
From Definition \ref{def-suka}
we have

\begin{eqnarray*}
\int_{B_{2R(X_0)}} |\nabla u|^{p-2}\na u(\na u-\na v)&=&\int_{B_{2R(X_0)}}\beta(u)(u_z- v_z),\\\nonumber
\int_{B_{2R(X_0)}}|\nabla v|^{p-2}\na v(\na u-\na v)&=&0.
\end{eqnarray*}
After subtracting the second equation from the first one we obtain 

\begin{eqnarray}\label{eq-10}
\int_{B_{2R(X_0)}} \left(|\nabla u|^{p-2}\na u-|\na v|^{p-2}\na v\right)(\na u-\na v)=\int_{B_{2R(X_0)}}\beta(u)(u_z-v_z).\\\nonumber
\end{eqnarray}
Recall that by Lemma 5.7 \cite{MZ} there is a generic constant $\mu>0$ depending only on $p$ and $N$ such that 
\begin{equation}\label{coercive}
(|\xi|^{p-2}\xi-|\eta|^{p-2}\xi)(\xi-\eta)\ge \mu\left\{
\begin{array}{lll}
|\xi-\eta|^p &\text{if}\ p>2,\\
|\xi-\eta|^2(|\xi|+|\eta|)^{p-2} &\text{if}\ 1<p\le 2, 
\end{array}
\right.
\end{equation}
for all $\xi, \eta\in \R^d$.
Hence 
\begin{eqnarray}
\mu \int_{B_{2R}{(X_0)}}|\na u-\na v|^p&\le& \frac{\e^p}p\int_{B_{2R}(X_0)}|\na u -\na v|^p+\frac1{\e^{p'}p'}\int_{B_{2R}(X_0)}|\beta(u)|^{p'}\\\nonumber
&\le& \frac{\e^p}p\int_{B_{2R}(X_0)}|\na u -\na v|^p+\frac1{\e^{p'}p'}(\ell+aM)^{p'}2^NR^N\omega_N\\\nonumber
\end{eqnarray}
where $M=\sup_{\C_L}|u|$.
Consequently, we get that 
\begin{eqnarray}
\int_{B_{2R}(X_0)}|\na u-\na v|^p\leq CR^N,
\end{eqnarray}
where 
$$C=\frac{(\ell+aM)^{p'}2^N\omega_N}{(\mu-\e^p/p){\e^{p'}p'}}.$$
We infer 
the estimate 
\begin{eqnarray}\label{est-RN-p}
\int_{B_{2R}(X_0)}|\na u-\na v|^2\leq CR^N,
\end{eqnarray}
with some tame constant $C>0$.

Indeed, as  $p>2$ we have by H\"older's inequality 
$$\left(\fint_{B_{2R}(X_0)}|\na u-\na v|^p\right)^{\frac1p}\geq 
\left(\fint_{B_{2R}(X_0)}|\na u-\na v|^2\right)^{\frac12}$$
and \eqref{est-RN-p} follows.

Furthermore, for any~$\rho>0$, we set 
$$ (\na u)_{X_0,\rho}:=\fint_{B_\rho(X_0)}\nabla u.$$
Then, from  H\"older's inequality we have 
\begin{equation}\begin{split}\label{starqew}
|(\na v)_{X_0,r}-(\na u)_{X_0,r}|^2 \le\,&
\left(\fint_{B_r(X_0)}|\nabla v-\nabla u|\right)^2\\
\le\,& \fint_{B_r(X_0)}|\nabla v-\nabla u|^2.
\end{split}\end{equation}

We would also need the following estimate for a $p-$harmonic function $v$:
there is $\alpha>0$ such that for all balls $B_{2R(X_0)}\Subset D$, 
with~$R\ge r>0$, there exists a universal constant $c>0$ 
such that the following Companato type estimate is valid
\begin{equation}\label{Companato-00}
\fint_{B_r(X_0)}|\na v-(\na v)_{X_0, r}|^2\leq c \left(\frac rR\right)^{\alpha}
\fint_{B_R(X_0)}|\na v-(\na v)_{X_0, R}|^2.
\end{equation}
See  \cite{DiB-M} Theorem 5.1. 

Denote $\|\cdot \|_{L^2({B_r(X_0)})}=\|\cdot \|_{2,r}$, then, 
using~\eqref{starqew}, we obtain 
\begin{eqnarray}\nonumber
\|\na u-(\na u)_{X_0,r}\|_{2,r}&\le& \|\na u-\na v\|_{2,r}+\|\na v-(\na v)_{X_0,r}\|_{2,r}\\\nonumber&&
+\,\|(\na v)_{X_0,r}-(\na u)_{X_0,r}\|_{2,r}\\\nonumber
&\le& 2\,\|\na u-\na v\|_{2,r}+\|\na v-(\na v)_{X_0,r}\|_{2,r}\\\label{Campanato}
&\le& 2\,\|\na u-\na v\|_{2,r}+C\left(\frac rR\right)^{\frac{N+\alpha}2}\|\na v-(\na v)_{X_0,R}\|_{2,R},\\\nonumber
\end{eqnarray}
where, in order to get \eqref{Campanato}, we used Campanato type estimate \eqref{Companato-00}.

From the triangle inequality for $L^2$ norm  we have
\begin{eqnarray*}
\|\na v-(\na v)_{X_0,R}\|_{2,R}\leq 2\,\|\na u-\na v\|_{2,R} +\|\na u-(\na u)_{X_0,R}\|_{2,R},
\end{eqnarray*}
and so, combining this with \eqref{est-RN-p}, we obtain 
\begin{eqnarray*}
\|\na u-(\na u)_{X_0,r}\|_{2,r}&\le& 2\,\|\na u-\na v\|_{2,r}\\\nonumber&&
+C\left(\frac rR\right)^{\frac{N+\alpha}2}\big[2\|\na u-\na v\|_{2,R} +\|\na u-(\na u)_{X_0,R}\|_{2,R}\big]\\\nonumber
&\le& C\left\{\|\na u-\na v\|_{2, R}+
\left(\frac rR\right)^{\frac{N+\alpha}2}
 \|\na u-(\na u)_{X_0,R}\|_{2, R}\right\}\\\nonumber
&\le& A\left(\frac rR\right)^{\frac{N+\alpha}2}
 \|\na u-(\na u)_{X_0,R}\|_{2, R}+B R^{\frac{N}2},\\\nonumber
\end{eqnarray*}
for some tame positive constants $A$ and $B$.

Introduce 
$$ \phi(r):=\sup\limits_{t\leq r} \|\na u-(\na u)_{X_0,t}\|_{2,t},
$$ 
then the former inequality can be rewritten as 
$$\phi(r)\le A\left(\frac rR\right)^{\frac{N+\alpha}2}\phi(R)+B R^{\frac N2},$$
with some positive constants $A, B, \alpha$. Applying
Lemma 2.1 from \cite{Giaq} Chapter~3, we  conclude that 
there exist~$R_0>0$ and~$c>0$ such that
$$\phi(r)\leq cr^{\frac N2}\left(\frac{\phi(R)}{R^{\frac N2}}+B\right),$$
for all $r\le R\le R_0$,
and hence
$$ \int_{B_r(X_0)}|\na u-(\na u)_{X_0,r}|^2 \le C r^N, $$
for some tame constant $C>0$.
This shows that~$\nabla u$ is locally BMO.
The log-Lipschitz estimate for $p>2$ now follows from 
\cite{Cianchi} Theorem 3.  
\end{proof}

\begin{rem}\label{non-var}
If $1<p<2$ then using the equation \eqref{eq-10} and inequality \eqref{coercive} in conjunction with the comparison of $u$ with its $p-$harmonic replacement  
in a small ball ${B_{2R}(X_0)}$ centred at a free boundary point,  renders the following inequality 
$$\int_{B_{2R}(X_0)}\frac{|\na u-\na v|^2}{\left(|\na u|+|\na v|\right)^{2-p}}\lesssim \int_{B_{2R}(X_0)}|u_z-v_z|\le \left(\int_{B_{2R}(X_0)}|\na u-\na v|^p\right)^{\frac1p}|B_{2R}|^{1-\frac1p}.$$
Consequently, setting $\sigma=\frac{p(2-p)}2>0$ and  applying H\"older's inequality we infer the estimate 
\begin{eqnarray*}
\int_{B_{2R}(X_0)}|\na u-\na v|^p&=&\int_{B_{2R}(X_0)}\frac{|\na u-\na v|^p}{\left(|\na u|+|\na v|\right)^{\sigma}}\left(|\na u|+|\na v|\right)^\sigma\\\nonumber
%%%%%% multiplying and dividing  
&\le& \left(\int_{B_{2R}(X_0)}\frac{|\na u-\na v|^2}{\left(|\na u|+|\na v|\right)^{\frac{2\sigma}p}}\right)^{\frac p2}
\left(\int_{B_{2R(X_0)}}(|\na u|+|\na v|)^{\frac{2\sigma}{2-p}}\right)^{\frac{2-p}2}\\\nonumber
%%%%%% repeating with sigma
&=& \left(\int_{B_{2R}(X_0)}\frac{|\na u-\na v|^2}{\left(|\na u|+|\na v|\right)^{2-p}}\right)^{\frac p2}
\left(\int_{B_{2R(X_0)}}(|\na u|+|\na v|)^p\right)^{1-\frac p2}\\\nonumber
&\lesssim&  \left( \left(\int_{B_{2R}(X_0)}|\na u-\na v|^p\right)^{\frac1p}|B_{2R}|^{1-\frac1p}\right)^{\frac p2}\left(\int_{B_{2R(X_0)}}|\na u|^p\right)^{1-\frac p2}
\end{eqnarray*}
which yields the estimate 
\begin{eqnarray*}
\int_{B_{2R}(X_0)}|\na u-\na v|^p\lesssim |B_{2R}|^{{p-1}} \left(\int_{B_{2R(X_0)}}|\na u|^p\right)^{2-p}.
\end{eqnarray*}
Then the Caccioppoli type inequality from \cite{Gary-0} and the technique above 
give that $u$ is H\"older continuous.
\end{rem}

%%%%%%%%%%
%%         SECTION
%%%%%%%%%%
\section{Proof of Theorem \ref{Lip}}\label{sec-thm2}
{$\bf 1^\circ$} We show that for any compact set $K\subset\subset\C_L$
there exists a tame constant $C$, depending on $\dist(K, \p\C_L)$
such that
\begin{eqnarray*}
 \sup_{B_{2^{-k-1}}(X)} u\leq \max\left(C2^{-k}, \sup_{B_{2^{-k}}(X)} u\right),
\qquad \forall X\in K\cap\partial\{u>0\}.
\end{eqnarray*}

If this inequality is false then
there exist a sequence of weak solution $u_j$ such that $0\leq u_j\leq M$ for some
fixed constant $M>0$, a sequence $\{k_j\}\subset \mathbb N,
X_j\in K\cap \partial\{u_j>0\} $
such that 
\begin{eqnarray}\label{est}
 \sup_{B_{2^{-k_j-1}}(X)} u_j > \max\left(j2^{- k_j}, \frac 12 \sup_{B_{2^{-k_j}}(X_j)} u_j\right).
\end{eqnarray}
%%%%
%%%%
Consider  the scaled functions $$\ds v_j(X)=\frac{u_j(X_j+2^{-k_j}X)}{S_j},$$
where $S_j=\ds \sup_{B_{2^{-(k_j+1)}}(X_j)}u_{\rred j}.$ It is obvious that 
\begin{equation}\label{eq-zero}
v_j(0)=0, 
\end{equation}
Moreover, it follows from
(\ref{est}) that \begin{eqnarray}\label{est-1}
                 \frac{2^{-k_j}}{S_j}<\frac1j, \qquad \sup_{B_{\frac12}} v_j \geq \frac12, \qquad 0\leq v_j(X)\leq 2,\ \  X\in B_1.
                 \end{eqnarray}
%%%%
Since, by assumption, the weak solutions $u_j$ are bounded it follows from
(\ref{est}) that $M>j2^{-k_j}$ implying that $k_j\rightarrow \infty.$

\smallskip

If $u_j$ solves  (\ref{e-balance}), then from the scale invariance properties of $\Delta_p$ it follows that 
$v_j$ solves the following equation
\begin{eqnarray}\label{scl-pde}
 \div(|\na v_j|^{p-2} \D v_j)&=&\frac{2^{-pk_j}}{S_j^{p-1}} ({\Delta}_p u_j)(X_j+2^{-k_j}X)\\\nonumber
 &=&
\left[\frac{2^{-k_j}}{S_j}\right]^{p-1}\div [\upbeta(v_j)\h v(X_j+2^{-k_j}X)]+f_j\\\nonumber
&\equiv&\div \h F_j+f_j,
\end{eqnarray}
%%%%
where 
\begin{eqnarray*}
\h F_j&=&\left[\frac{2^{-k_j}}{S_j}\right]^{p-1} \upbeta(v_j)\h v(X_j+2^{-k_j}X), \\\nonumber
f_j&=&\frac{2^{-pk_j}}{S_j^{p-1}} f(X_j+2^{-k_j}X)=S_j\left[\frac{2^{-k_j}}{S_j}\right]^p f(X_j+2^{-k_j}X).
\end{eqnarray*}

From  $\h v\in \h L^\infty(\C_L, \R^N)$ we obtain,  using   (\ref{est-1}), definition of
$S_j$ and (\ref{beta}),
the inequality
$$|\h F_j|\leq \left[\frac{2^{-k_j}}{S_j}\right]^{p-1}
\upbeta(2)\sup|\h v|\leq \left[\frac{1}{j}\right]^{p-1}\upbeta(2)\sup|\h v|\rightarrow 0.$$
Similarly we obtain $\sup\limits_{B_1}|f_j(X)|\rightarrow 0$.

\smallskip

From the Caccioppoli inequality (see Appendix) it follows
that $\{v_j\}$ is bounded in $W^{1,p}(B_{\frac34})$. 
Furthermore, utilizing  (\ref{scl-pde}), \eqref{est-1} and Serrin's 
theorem for quasilinear divergence form elliptic operators \cite{Serrin}, 
we infer  that  the sequence $\{v_j\}$ is uniformly H\"older continuous in
$B_{3/4}$. 
%%%%
Now employing  a customary compactness argument and the
estimates for $\{\h F_j\}$ and $\{f_j\}$,  we can extract a subsequences $j_m$ such that $X_{j_m}\to X_0$,
$\{v_{j_m}\}\subset \{v_j\}$ which uniformly converges to some
$v_0$
in $B_{\frac34}$ and weakly in $W^{1,p}(B_\frac 34)$. Moreover, we can check that 
$$-\int |\na v_0|^{p-2}\D v_0 \D \phi\longleftarrow -\int |\D v_{j_m}|^{p-2}\D v_{j_m} \na \phi=\int f_{j_m}\phi-\h F_{j_m}\cdot \na\phi\longrightarrow 0,
\qquad \forall \phi \in C_0^\infty(B_{\frac 34}).$$
To see this we first prove 
\begin{clm}\nonumber
For every $q\ge p$ there is a tame constant $\gamma$ independent of $j_m$ such that 
\[\|\nabla v_{j_m}\|_{L^{q}(B_{\frac34})}\le \gamma, \quad \|\nabla v_0-\nabla v_{j_m}\|_{L^{q}(B_{\frac34})}\to 0.\]
\end{clm}
\begin{proof}
\bblue{Indeed, let us define $h_j\in W^{1, 2}_0(B_1)$ as the solution of the following homogeneous Dirichlet problem 
\[\Delta h_j=f_j\quad \hbox{in}\ B_1, \quad h_j=0\ \ \hbox{on} \ \partial B_1.\] 
From the a priori bound (3.12) in \cite{GT} it follows 
that $$\sup_{B_1}|h_j|\le C\sup_{B_1}|f_j|$$
with some tame constant $C>0$.
Observe that   $h_j$ is the convolution of $f_j$ and the Green function of $B_1$. 
Using the estimates for the Green potentials and the fact that $\sup\limits_{B_1}|f_j(X)|\rightarrow 0$
it follows that 
\[\|h_j\|_{C^{1, \sigma}(B_{1})}\to 0\ \ \hbox{as}\ j\to \infty\]
for any $\sigma\in (0, 1)$, see estimates (4.45) and (4.46) in \cite{GT}. Let $\h H_j=\h F_j+\nabla h_j$ and define 
$\mathcal F_j=|\h H_j|^{\frac1{p-1}-1}\h H_j$ then 
\[\Delta_p v_{j}=\div \h F_j+f_j=\div \h H_j=\div(|\mathcal F_j|^{p-2}\mathcal F_j)\]
and $\mathcal F_j\in L^{\infty}(B_1)$. In fact, 
\[\|\mathcal F_j\|_{L^{\infty}(B_{1})}\le\left(\|\h H_j\|_{L^\infty(B_1)}\right)^{\frac1{p-1}}\le\left(\|\h F_j\|_{L^\infty(B_1)}+\|\nabla h_j\|_{L^\infty(B_1)}\right)^{\frac1{p-1}}\to 0\quad \hbox{as}\ j\to \infty.\]
Now we are in position to apply Theorem 1.2 \cite{DiB-M} to infer that for any $q\ge p$ the following estimate is true 
\[\|\nabla (\zeta v_j)\|_{L^q(\R^N)}\le \gamma(\|\zeta \mathcal F_j\|_{L^q(\R^N)}+\|\zeta v_j\|_{L^{p}(\R^N)})\]
where $\zeta\in C_0^\infty(B_1)$ is a cut-off function and $\gamma$ is a tame constant. This in particular yields
\[\|\nabla v_{j_m}\|_{L^{q}(B_{\frac34})}\le \gamma, \quad \|\nabla v_0-\nabla v_{j_m}\|_{L^{q}(B_{\frac34})}\to 0\]
for any $q\ge p$ and a suitable subsequence $j_m$ with some tame constant  $\gamma$ independent of $j_m$, because the weak convergence of the gradients in 
$L^{q'}$ implies strong convergence in $L^{q}$ if $q'>q$. This finishes the proof of the claim.}
\end{proof}
%%%%%
\bblue{
It remains to note that 
\begin{eqnarray*}
\left|\int_{B_{\frac34}} \left(|\na v_0|^{p-2}\D v_0 - |\D v_{j_m}|^{p-2}\D v_{j_m}\right) \na \phi \right|&\le& \sup |\na \phi|\int_{B_{\frac34}} \left||\na v_0|^{p-2}\D v_0 - |\D v_{j_m}|^{p-2}\D v_{j_m}\right| \\\nonumber
&\le& \gamma(p) \sup |\na \phi|\int_{B_{\frac34}} \left|\D v_0 - \D v_{j_m}\right|^{p-1}
\end{eqnarray*}
provided that $1<p\le 2$, where the last estimate follows from  \cite{Lindqvist} page 43. 
Note that  
\[\int_{B_{\frac34}} \left|\D v_0 - \D v_{j_m}\right|^{p-1}\le 
\left(\int_{B_{\frac34}} \left|\D v_0 - \D v_{j_m}\right|^p \right)^{\frac{p-1}p} |B_{\frac34}|^{\frac1p}\to 0.\]
As for the case $p>2$ again from \cite{Lindqvist} page 43 by choosing $q> 2(p-2)$ we have 
\begin{eqnarray*}
\int_{B_{\frac34}} \left||\na v_0|^{p-2}\D v_0 - |\D v_{j_m}|^{p-2}\D v_{j_m}\right| &\le& 
(p-1)\int_{B_{\frac34}}|\nabla v_0-\na v_{j_m}|(|\na v_0|^{p-2}+|\na v_{j_m}|^{p-2})\\\nonumber
&\lesssim&\left(\int_{B_{\frac34}}|\nabla v_0-\na v_{j_m}|^2\int_{B_{\frac34}}|\na v_0|^{2(p-2)}+|\na v_{j_m}|^{2(p-2)}\right)^{\frac12}\\\nonumber
&\to&0\ \ \hbox{as}\ j_m\to \infty.
\end{eqnarray*}
}
%%%%%%%%%%
Thus $v_0\in W^{1,p}(B_{\frac34})$ is a nonnegative continuous  solution of $\div(|\D v_0|^{p-2}\D v_0)=0$ 
in $B_{\frac 34}$. On the other hand, it follows  from
uniform convergence $v_{j_m}\rightarrow v_0$ that  (\ref{eq-zero}) translates to $v_0$ and we have $v_0(0)=0$ and 
also $\sup\limits_{B_{\frac 12}} v_0=\frac 12$ thanks to \eqref{est-1}.
However this is in contradiction with  the strong maximum principle  and the proof 
\rred{of linear growth from the free boundary} follows.

\medskip 

{$\bf 2^\circ$}
Let us take $X_0\in \C_L$ such that $u(X_0)>0$ and set $r=\dist(X_0, \partial\{u>0\})$.
Then by linear growth $u(X_0)\leq C r$ with some tame constant $C>0$.
Consider the scaled function $v(X)=\frac{u(X_0+rX)}r, X\in B_1$. Then $v\ge 0$ solves the equation 
$$\div(|\D v(X)|^{p-2}\nabla  v(X))=\div(\h F(X_0+rX))+rf(X_0+rX)\quad \text{ in}\  B_1$$ 
and $v(0)\leq C$.
Here $\h F(X)=au(X) e_N\in C^\alpha(B_1)$ by Lemma \ref{lemma:coherence}.
From the weak Harnack inequality, \cite{Serrin} Theorem 7,  we have that 
\begin{eqnarray*}
\sup_{B_{\frac12}}v &\leq& c_0(v(0)+\|\h F\|_{L^\infty(B_1)}^{\frac1{p-1}}+(r\|f\|_{L^\infty(B_1)})^{\frac1{p-1}})\\
&\leq& c_0(C+\|\h F\|_{L^\infty(B_1)}^{\frac1{p-1}}+(r\|f\|_{L^\infty(B_1)})^{\frac1{p-1}}).
\end{eqnarray*}

From the local gradient estimates \cite{Lieberman}  we infer that 
$$\sup\limits_{B_{1/4}}|\nabla  v|\leq 
c_0(C+\|\h F\|_{L^\infty(B_1)}^{\frac1{p-1}}+(\|\h F\|_{C^\alpha(B_r)}^*)^{\frac1{p-1}}+(r\|f\|_{L^\infty(B_1)})^{\frac1{p-1}})$$ 
where the star above means the classical weighted H\"older norm using the radius of the ball. In particular 
$|\nabla v(0)|\leq \widetilde C$ for some tame constant $\widetilde C$ and rescaling back we infer $|\nabla u(X_0)|\leq \widetilde C$. This 
completes  the proof of Theorem 2. 
\qed

%%%%%%%%%%
%%         SECTION
%%%%%%%%%%
\section{Regularity of free boundary, Proof of Theorem \ref{FB}} \label{sec-thm3}

\subsection{Proof of Proposition \ref{cor-1}}
We recast Proposition \ref{cor-1} here.
\begin{lem}\label{lem-noneg}
Let $u\ge 0$ be a weak solution of \eqref{DP} in the sense of Definition \ref{DP-def}, $N=2<p<\infty, m^-=0$ 
and assume further that
\begin{equation}%\label{g-cond}
\liminf_{z\to z_0}\frac{g(x,z)-g(x, z_0)}{z-z_0}\ge 0, \quad \forall x\in \p\Om, z_0\in [0,L], \quad \p_z g(X)=0, X\in \p \Sigma
\end{equation}
where $g\in W^{2, 2+\eta_0}(\Sigma), \eta_0>0,$ and 
$\Sigma$ is the lateral boundary of $\C_L$.
Then $u$ is monotone nondecreasing in $z$ direction.
%%%%
%Then we have $\p_zu(X)\ge0$ for all $X\in \{u>0\}\cap \C_L$.
\end{lem}
\begin{proof}
For $\e>0$ small let us consider the mollified problem

\begin{equation}\label{mol-pde}
\div\left((\e^2+|\na u^\e(X)|^2)^{\frac{p-2}2}\na u^\e(X)\right)=\partial_z(\beta(u^\e(X))), \qquad X\in \C_L,
\end{equation}
with boundary condition $u^\e(X)=g(X)$ on $\partial \C_L$, where 
$g\in W^{2, 2+\eta_0}(\Sigma), \eta_0>0$
is satisfying 
\eqref{2}. The existence of $u^\e$ for each $\e>0$ follows from standard penalisation argument
for uniformly elliptic equations.

%% Claim W22+\eta
\begin{clm}\label{clm-11}
Let  $\delta>0$ be small. Then  $u^\e\in W^{2, 2+\eta} (\overline{ \{u^\e>\delta\} \cap \C_L})$ for some 
$\eta>0$ which depends on $\e$ and $\|g\|_{W^{2, 2+\eta_0}}.$ 
\end{clm}
\begin{proof}
In order to prove this claim we first extend $u^\e$ to the cylinder $(0, \frac{3L}L)\times \Om:=\widetilde\C_L$ such that the extended function $\widetilde u^\e$ solves the equation \eqref{mol-pde}. 
Introduce the upper extension of $g$ as follows 
\[
\widetilde g(x, z)=\left\{\begin{array}{ll}
g(x, z) &\hbox{if}\ x\in \p\Om, z\in(0, L),\\
g(x, 2L-z) &\hbox{if}\ x\in \p \Om, z\in(L, \frac{3L}2).
\end{array}
\right.
\]
Since $g(x, L)=m^+=const, x\in \p\Om$ and by assumption \eqref{g-cond} $\p_zg(x, L)=0, x\in \p\Om$ it follows that 
$\wt g\in W^{2, 2+\eta_0}(\wt \C_L)\cap C^{1, \alpha_0}(\wt \C_L)$ where $\alpha_0=\frac{\eta_0}{2+\eta_0}$. This can be seen  from the embedding of the Sobolev space $W^{2, 2+\eta_0}(\wt \Sigma)$ where 
$\wt \Sigma $ is the lateral boundary of $\wt \C_L$ (recall that $N=2$).
The upper extension of $u^\e$ are defined accordingly
\[\widetilde u^\e(x, z)=\left\{\begin{array}{ll}
u^\e(x, z) &\hbox{if}\ x\in \Om, z\in(0, L],\\
u^\e(x, 2L-z) &\hbox{if}\ x\in \Om, z\in(L, \frac{3L}2).
\end{array}
\right.
\]
 Let us check that $\wt u^\e$
is a solution of \eqref{mol-pde} across $\Om\times\{L\}$. 
Note that by the continuity of $u^\e$, which follows from  Theorem \ref{thm-bmo} and Remark \ref{non-var}, 
near $\Om\times\{L\}$ $v=u^\e-m^+$ solves the equation $\div\left((\e^2+|\na v|^2)^{\frac{p-2}2}\na v\right)=a\p_z v$ and $v=0$ on $\Om\times\{L\}$. By Proposition 1 \cite{Tolksdorf} it follows that 
$u^\e\in W^{2,2}_{loc}(\C_L).$ 

Take $\phi\in C_0^\infty(\wt \C_L)$ such that for some 
ball $B$ we have $\hbox{supp }\phi\subset  B\subset \wt \C_L$. Denote $D^-=B\cap \C_L$ and $D^+=B\cap( \Om\times(L, \frac{3L}2))$ and fix $t>0$ small. Then from the divergence theorem we get 
\begin{eqnarray*}
I^+_t&:=&\int_{D^-\cap \{z<L-t\}} \left((\e^2+|\na \wt u^\e|^2)^{\frac{p-2}2}\na \wt u^\e\right)\cdot \na\phi\\\nonumber
&=&
\int_{\p D^-\cap \{z<L-t\}} \phi (\e^2+|\p_z u^\e|^2)^{\frac{p-2}2}\p_z u^\e-
\int_{D^-\cap \{z<L-t\}} \phi \div\left((\e^2+|\na u^\e|^2)^{\frac{p-2}2}\na u^\e\right)\\\nonumber
&=&
\int_{\p D^-\cap \{z<L-t\}} \phi (\e^2+|\p_z u^\e|^2)^{\frac{p-2}2}\p_z u^\e-
\int_{D^-\cap \{z<L-t\}} \phi a\p_z u^\e
\end{eqnarray*}
where the last equality follows from the $W^{2,2}_{loc}(\C_L)$ estimates mentioned above. 
Similarly we have that 
\begin{eqnarray*}
I^+_t
&:=&\int_{D^+\cap \{z>L+t\}} \left((\e^2+|\na \wt u^\e|^2)^{\frac{p-2}2}\na \wt u^\e\right)\cdot \na\phi\\\nonumber
&=&-\int_{\p D^+\cap \{z>L+t\}} \phi (\e^2+|\p_z \wt u^\e|^2)^{\frac{p-2}2}\p_z \wt u^\e-
\int_{D^-\cap \{z>L+t\}} \phi a\p_z \wt u^\e.
\end{eqnarray*}
From the gradient estimates near the flat portions 
of the boundary \cite{Lieberman} it follows that $\na u^\e$ is H\"older continuous near $\Om\times\{L\}$,  therefore
\[\lim_{t\to 0}\left(\int_{\p D^-\cap \{z<L-t\}} \phi (\e^2+|\p_z u^\e|^2)^{\frac{p-2}2}\p_z u^\e-\int_{\p D^+\cap \{z>L+t\}} \phi (\e^2+|\p_z \wt u^\e|^2)^{\frac{p-2}2}\p_z \wt u^\e\right)=0.\]
Consequently 
\begin{eqnarray*}
\int \left((\e^2+|\na \wt u^\e|^2)^{\frac{p-2}2}\na \wt u^\e\right)\cdot \na\phi= \int\phi a\p_z\wt u^\e. 
\end{eqnarray*}
Near the lateral boundary of $(\Omega\times(0,\frac{3L}2))\cap\{\wt u^\e>\delta\}$ we can take 
$\phi=(\wt u^\e-g)\zeta, \zeta\in C_{0}^\infty$ such that 
$\phi\in W^{1, p}_0(\wt \C_L)$. The finite differences of $\na \wt u^\e$
in  the $z$ variable (for sufficiently small step size compared to $\delta$) satisfy uniform  $W^{2,2}$ estimates, see the proof of Lemma 8.12 in \cite{GT}  (recall that by assumption $N=2$ and hence the lateral boundary is flat).
%(via change of variables and odd reflection principle it reduces to the interior estimates thanks to the $C^{1, \alpha}$ smoothness of $\partial\Omega$) 
Thus $\p_{zz}\wt u^\e, \p_{xz}\wt u^\e \in W^{2, 2} (\overline{ \{u^\e>\delta\} \cap \C_L})$
and $\p_{xx} \wt u^\e\in W^{2, 2} (\overline{ \{u^\e>\delta\} \cap \C_L})$ follows directly from the equation.

Differentiating the equation $\div\left((\e^2+|\na u^\e|^2)^{\frac{p-2}2}\na u^\e\right)=a\partial_z u^\e$
in $\{u^\e>0\}$ we get 
\begin{equation}\label{lem-z-der}
\div\left((\e^2+|\na u^\e|^2)^{\frac{p-2}2}\left[id+(p-2)\frac{\na u^\e\otimes\na u^\e}{\e^2+|\na u^\e|^2}\right]\na \partial_z u^\e\right)=a\partial_z(\partial_zu^\e).
\end{equation}
For each $\e>0$ the matrix 
$$A(x)=(\e^2+|\na u^\e|^2)^{\frac{p-2}2}\left[id+(p-2)\frac{\na u^\e\otimes\na u^\e}{\e^2+|\na u^\e|^2}\right]$$
is strictly elliptic. 
Hence $u^\e \in W^{2, 2+\eta} (\overline{ \{u^\e>\delta\} \cap \C_L})$ follows from the strict ellipticity of $A(x)$ and a standard application of Gehring's Lemma, see 
\cite{Giaq} Theorem 2.1, page 136.  This finished the proof of the claim.
\end{proof}
%%%%%%%% end claim 

%% Min principle
In view of \eqref{lem-z-der} the function 
$w=\p_z u^\e$ solves the equation $\div(A(x)\na w)=a\p_zw$ with strictly elliptic matrix $A$.
Hence from minimum principle and it follows that 
$$ 
\min\limits_{\partial(\{u^\e>0\}\cap \C_L)}\partial_z u^\e=\min\limits_{\{u^\e>0\}\cap \C_L}\partial_zu^\e
.$$

Take  an arbitrary $\delta>0$ small and let us show that $\min\limits_{\partial(\{u^\e>\delta\}\cap \C_L)}\partial_z u^\e\ge 0$.
Applying Claim \ref{clm-11} we  have that $u^\e\in W^{2, 2+\eta} (\overline{ \{u^\e>\delta\} \cap \C_L})$ for some 
$\eta>0$ which depends on $\e$ and $\|g\|_{W^{2, 2+\eta_0}}.$ 
%This follows from the global $W^{2,2}$ estimates \cite{GT} (via change of variables and odd reflection principle it reduces to the interior estimates thanks to the $C^{1, \alpha}$ smoothness of $\partial\Omega$) for the uniformly elliptic equations and application of Gehring's Lemma, see \cite{Giaq} Theorem 2.1, page 136.  

From Sard's theorem \cite{Figalli} it follows that 
the one dimensional Lebesgue measure of the critical values of $u^\e$ is zero. 
Consequently, $\partial\{u^\e>\delta\}$ is a regular  curve for a.e. $\delta>0$ and the trace of $u^\e_z$ is well defined on it.

That said, let us consider the following cases:
\begin{itemize}
\item $X_0=(x_0, z_0)\in \partial\{u^\e>\delta\}\cap \C_L$ then for $(x, z_0)\in \fb u\cap \overline{\C_L}$ the lower bound follows immediately  
$$\liminf_{\begin{subarray}{c}z\to z_0\\z<z_0\end{subarray}}\frac{u(x_0, z)-u(X_0)}{z-z_0}=\liminf_{\begin{subarray}{c}z\to z_0\\z<z_0\end{subarray}}\frac{u(x_0, z)-\delta}{z-z_0}\ge 0.$$
\item $X_0 \in \partial\{u^\e>\delta\}\cap\p\left(\Omega \times(0,L)\right)$ then 
on the lateral boundary the tangential derivative 
agrees with that of $g$. 
\item $X_0\in \overline{\Omega}\times\{L\}$ it follows from Hopf's lemma,  because by the maximum principle 
$m^+=\max u^\e$ and $\na u^\e$ is H\"older continuous near $\Om\times\{L\}$ (see the reflection argument in the proof of Claim 11 and the application of boundary gradient estimates from \cite{Lieberman}) hence $\p_z u^\e>0$.
\end{itemize}
Consequently we conclude that $\min\limits_{\partial(\{u^\e>0\}\cap \C_L)}\partial_z u^\e\ge 0$
because $\delta>0$ was arbitrary small number.
Since the solution $u$ is unique, in view of Remark \ref{Holder}, then $u^\e\to u$ weakly in $W^{1, p}(\C_L)$ and thus
$$0\le \partial_z u(X), \qquad X\in \{u>0\}\cap \C_L.$$
\end{proof}

\subsection*{Proof of Theorem \ref{FB}} 
Let $N=2$ and \eqref{g-cond} holds,  we show that 
the free boundary is a continuous curve over $\Omega$.

 For $x\in \Omega$ introduce the following height functions 
\begin{eqnarray}
h^+(x)&=&\sup\{z\ s.t.\ u(x, z)=0\},\\\nonumber
h^-(x)&=&\inf\{z\ s.t. \ u(x, z)=0\}.\\\nonumber
\end{eqnarray}

If $h^+(x_0)>h^-(x_0)$ for some $x_0\in \Omega$ then 
it follows from $\partial_z u\ge 0$ (see Lemma \ref{lem-noneg}) that 
the free boundary contains a vertical segment of the form 
$I_0=\{x_0\}\times(a, b)$ for some $a<b$. On $I_0$ we have that 
$\partial_z u=0$. On the other hand the free boundary condition \eqref{Stefan-cond}  is satisfied in 
the classical sense on $I_0$.
Hence $|u_x|^{p-2}u_x=0$ implying 
\begin{equation}\label{nodal}
u=|\na u|=0\qquad \hbox{on}\  I_0.
\end{equation} 
However, since there is a 
touching ball from $\{u>0\}$ at the points on $\{x_0\}\times(a+\e, b-\e)$ for small $\e>0$ then it follows from 
Hopf's lemma (see the Appendix) that $|\na u|\not =0$ which is in contradiction with \eqref{nodal}.

%%%%%%%%%%
%%         SECTION
%%%%%%%%%%
\section{Behaviour of $\fb u$ near the fixed boundary}\label{sec-gurev}
In this section we show that at the contact points $\fb u\cap \Sigma$ 
$u$ grows at most linearly and this is contained in Lemma \ref{lem-blya}. 
We begin with a simple observation 
\begin{lem}\label{lem-subsolution}
If
$u$ is a weak solution of \eqref{mainpde} with $u_z\ge 0$ then $u$ is a weak solution of the 
following differential inequality
\begin{equation}\label{subsolution}
\Delta_pu-au_z\ge 0 \ \ \hbox{in}\ \C_L.
\end{equation}
\end{lem}
\begin{proof}
Indeed, for $\zeta\ge 0, \zeta\in C_0^\infty(B)$, with some ball $B\subset \C_L$ it follows that 
\begin{eqnarray}
0&=&\int_{\C_L}|\nabla u|^{p-2}\na u\na\zeta -\int_{\C_L}\beta(u)\zeta_z \\\nonumber
&=&
\int_{\C_L}|\nabla u|^{p-2}\na u\na\zeta -\int_{\C_L}a u\zeta_z-\int_{\C_L}\ell\chi_{\{u>0\}}\zeta_z.
\end{eqnarray}
Thus it is enough to show that  $\int_{\C_L}\chi_{\{u>0\}}\zeta_z\le 0$.
Let us choose a sequence $\gamma_k(t), t\in \R$ such that 
$\gamma'_k(t)\ge 0, r\in \R$ and $\gamma_k\to \chi_{\{t>0\}})$ weak star.  
It follows that 
\begin{eqnarray}
\int \chi_{\{u>0\}}\zeta_z&=&\lim_{k\to \infty}\int \gamma_k(u)\zeta_z\\\nonumber 
&=&-\lim_{k\to \infty}\int_{B}\gamma_k'(u)u_z\zeta\\\nonumber
&=& -\lim_{k\to \infty}\int_{B}\gamma_k'(u)(u^+_z-u^-_z)\zeta\\\nonumber
&=&  -\lim_{k\to \infty}\int_{B}\gamma_k'(u)u^+_z\zeta\\\nonumber
&\le& 0
\end{eqnarray}
where we used the notation  $v^+=\max (v,0), v^-=-\min(v,0)$ and 
the last line follows from Lemma \ref{lem-noneg}.
\end{proof}

The next lemma is true in all dimensions $N\ge2$.

\begin{lem}\label{Harnack}
Let $p>2$ and $u\ge 0$ be a weak solution to $\Delta_pu=\partial_z\beta(u)$ in $\C_L$. Then there is a
constant $C_0>0$ depending only on $\ell, a, N, p$ and $\|u\|_{L^\infty(\C_L)}$ such that 
$$|\na u(X)|\le C_0\qquad \forall X\in \fb u\cap \C_L.$$
\end{lem}
\begin{proof}
Let $p>2$ then if $X\in \fb u\cap \C_L$ then $d(X)=\dist(X, \p\C_L)>0$. For 
$r<\frac{d(X)}2$ 
\begin{equation}\label{30}
\int_{B_r(X)}\left|\nabla u-\fint_{B_r(X)}\nabla u\right|^2\leq Cr^N
\end{equation}
where $C$ depends on $a, \ell, p, \|u\|_\infty$ and $N$, see Lemma \ref{lemma:coherence}.

Let $\e>0$ be small, for $\rho>\e$ and $X_0\in\fb u, B_{2\rho}(X_0)\subset \C_L$ we have 
\begin{equation}
\frac1{\rho^N}\int_{B_\rho(X_0)}u= \frac1{\e^N}\int_{B_\e(X_0)}u+\int_{\e}^\rho\frac{d}{dt}\left(\frac1{t^N}\int_{B_t(X_0)}u\right)dt.
\end{equation}
From the log-Lipschitz continuity of $u$ it follows that the Lebesgue Differentiation Theorem holds everywhere. Thus, it follows that
$$\lim_{\e \to 0}\frac1{\e^N}\int_{B_\e(X_0)}u=0$$
and consequently  
\begin{eqnarray}\label{lin-grth}
\frac1{\rho^N}\int_{B_\rho(X_0)}u&=& \int_{0}^\rho\frac{d}{dt}\left(\frac1{t^N}\int_{B_t(X_0)}u\right)dt\\\nonumber
&=& \int_{0}^\rho\frac{d}{dt}\left(\int_{B_1}u(X_0+tX)dX\right)dt\\\nonumber
%------
&=&\int_{0}^\rho\left(\frac1{t^N}\int_{B_t}\na u(X_0+tX)\frac{X}{t}dX\right)dt\\\nonumber
&=&\int_{0}^\rho\left(\frac1{t^N}\int_{B_t}\na u(X_0+Y)\frac{Y}{t}dY\right)dt\\\nonumber
&=& \int_{0}^\rho\left(\frac1{t^N}\int_{B_t(X_0)}\left[\na u(X)-\fint_{B_t(X_0)}\na u\right]\frac{X-X_0}{t}dX\right)dt \\\nonumber
&\le& \int_{0}^\rho\left(\frac1{t^N}\int_{B_t(X_0)}\left|\na u(X)-\fint_{B_t(X_0)}\na u\right|dX\right)dt \\\nonumber
&\le& C\rho
\end{eqnarray}
where the last line follows from \eqref{30}.
It remains to apply Harnack's inequality in order to finish the proof. 
Let $v(X)=\frac{u(X_0+\rho X)}{\rho}, X\in B_1$ then 
$\Delta_p v(X)=\partial_z(\beta(u(X_0+\rho X)))$. By Lemma \ref{lem-subsolution} 
$$\div(|\na v|^{p-2}\na v)\ge a\rho v_z$$
Denote $\cal A(v, \xi)=|\xi|^{p-2}\xi%-\ell\chi_{\{v>0\}}
, \cal B(\xi)=a\rho\xi_N$.
Thus $v$ solves an inequality of the following form $\div \cal A(v, \na v)\ge \cal B(\na v)$ and
\begin{eqnarray*}
|\cal A(v, \xi)|&\le& |\xi|^{p-1}%+\ell
\\
|\cal B(\xi)|&\le& \frac{|\xi|^p}{p}+\frac{(a\rho)^{p'}}{p'}\\
\cal A(v, \xi)\cdot \xi&\ge& \frac{p-1}p|\xi|^p%-\frac{\ell^{p'}}{p'}\\
\end{eqnarray*}
where $1/p+1/{p'}=1$.
Thus $\cal A$ and $\cal B$ satisfy the 
structural conditions (3.5) in \cite{MZ}. 
From the weak Harnack inequality, Corollary 3.10 \cite{MZ} we infer that 
$$\sup_{B_{\frac12}}v \le C\left[\fint_{B_1}v+\kappa \right]$$
where $C$ depends on $p, N$ and $\kappa$ depends on $p,\ell, N$ and $a$.
Now the desired estimate follows from \eqref{lin-grth}. 
\end{proof}

\smallskip 

Observe that Lemma \ref{Harnack} is stronger than Theorem \ref{Lip}
since the  constant $C_0$  does not depend on the distance of the point $X_0$ from $\Sigma$.

\begin{lem}\label{lem-blya}
Let $u\ge 0$ be as in  Theorem \ref{FB}, $N=2<p$.
 There is a constant $\mu>0$ such that 
\begin{equation}
u(X)\le\mu|X-X_0|
\end{equation}
for any 
$X_0\in \fb u \cap \Sigma.$
\end{lem}

\begin{proof}
Let $X_0\in \partial \Omega\times(0,L)$ and take $r$ small, say $r=\hbox{diam}{\Omega}/100$ such that 
$B_{r}(X_0)\cap \C_L$ is a half ball entirely in $\C_L$. Recall that  $N=2$ and therefore the lateral boundary of 
$\C_L$ is flat. 
Let $w$ be the solution to the following Dirichlet problem
\[\left\{
\begin{array}{lll}
\Delta_pw-aw_z=0 & \hbox{in}\ \ \ C_L\cap B_r(X_0),\\
w=u & \hbox{on}\ \ \partial (\C_L\cap B_r(X_0)).
\end{array}
\right.
\]
Since $g\in {W^{2,2+\eta_0}}(\Sigma)$ then we can apply the interior gradient estimates  from \cite{Lieberman} in the half ball 
$B^+_r=\C_L\cap B_r(X_0)$ to infer the following estimate
\begin{equation}\label{grad-est-hav}
\sup_{B_{r/2}}|\na w|\le C\frac{\sup_{B_{r}^+}w}{r}=C\frac{\sup_{\p B_{r}^+}w}{r}\le C(r)
\end{equation}
where $C(r)$ also depends on $\|g\|_{W^{2,2+\eta_0}}$.
By \eqref{subsolution} $u$ is a subsolution and hence we can apply the  comparison principle  Theorem 3.5.1
\cite{Serrin-book} and the
boundary gradient estimate \eqref{grad-est-hav} in order to obtain 
$u\le w\le C|X-X_0|$  in $B_{\frac r2}(X_0)$ with tame constant $C>0$ depending only on 
$a, \ell, \|g\|_{W^{2,2+\eta_0}}, r, p$.
\end{proof}

%%%%%%%%%%
%%         SECTION
%%%%%%%%%%
\section{Concluding remarks}\label{gens}
When the governing quasilinear equation is uniformly elliptic then 
the arguments can be considerably simplified.
As an example let us consider the following operator \cite{ACF-quasi}: 
Let $F(t)$ be a function in $C^{2, 1}[0, \infty)$ satisfying 
\begin{eqnarray}
F(0)=0, \quad c_0\le F'(t)\le C_0,\\\nonumber
0\le F''(t)\frac{C}{1+t},
\end{eqnarray}
for some positive constants $c_0, C_0$.
From here we find that 
$f(\xi)=F(|\xi|^2)$ is convex and the following holds 
\begin{eqnarray}
\gamma|\xi|^2\le f(\xi)\le \frac1\gamma|\xi|^2,\\\nonumber
f_\xi(\xi)\cdot \xi\ge\gamma|\xi|^2,\\\nonumber
\gamma|\eta|^2\le\sum_{ij}\frac{\partial^2 f(\xi)}{\partial \xi_i\partial \xi_j}\eta_i\eta_j\le \frac1\gamma|\eta|^2,
\end{eqnarray}
where $f_\xi=\nabla_\xi f$ and $\gamma>0$.
The quasilinear operator 
$\mathcal L v=\div(f_\xi(\na v))$ is now uniformly elliptic and Theorems \ref{thm-bmo}-\ref{Lip}
can be extended to the solutions of $\cal Lu=\partial_z(\beta(u))$ with less efforts. 
Moreover, by differentiating $\cal Lu=au_z$ in $z$-direction we can see that 
$u_z$ solves a strictly elliptic operator and hence applying the comparison principle (which can be proved by a standard method discussed in \cite{CY})
one can infer that $u_z\ge 0$ provided that \eqref{g-cond} and $u\ge 0$ hold.

Notwithstanding its simple form the equation $\Delta_p u=\partial_z(\beta(u))$ differs drastically 
from its strictly elliptic counterpart. This in particular includes: 
\begin{itemize}
\item The strong comparison principle is not known for the $p-$Laplace structure as opposed to 
the case $p=2$, see \cite{CY} Lemma 2.1. The reason is that one cannot define the corresponding strictly elliptic
adjoint problem and hence Kamin's argument cannot be generalised directly.
\item When $p=2$ then one  can deduce that the solution $u\ge 0$ is non-degenerate at the free boundary 
points by using Biacchi's transformation. This allows to transform the continuous casting problem to 
an obstacle like problem and apply the techniques developed for the latter for a class of
divergence form elliptic equations \cite{K-DCDS}. This argument fails to 
work when $p\not=2$ due to the nonlinear structure of the operator $\Delta_p u$.

\end{itemize}

It would be intersting to find out whether any of these difficulties can be circumvented 
which will lead to stronger free boundary regularity.

\section*{Appendix}
\subsection{The Caccioppoli inequality}
\begin{lem}
Let $u$ be a weak solution of the equation $\Delta_p u=\partial_z(\beta(u))+f$ in $B_1$, $f\in C(\overline{\C_L})$.
Then there is a constant $\Gamma$ depending only on 
$\sup_{B_1}|u|, \sup_{\C_L}|f|, a, \ell, p$ and $N$ such that 
\begin{equation}
\int_{B_{\frac34}}|\na u|^p\le \gamma. 
\end{equation}

\end{lem}
\begin{proof}
\begin{equation}
\int |\D u|^{p-2}\na u\na \zeta=\int\beta(u)\partial_z\zeta-\int f\zeta.
\end{equation}
Let $\zeta=u\phi^p$ where $\phi\ge0$ in $B_1$, $\phi=0$ in $\R^N\setminus B_1$, $\phi=1$ in $B_{\frac34}$, 
and $|\na \phi|\le C$ with some tame constant $C>0$.
Then we have 
\begin{eqnarray}
\int|\na u|^p\phi^p+|\na u|^{p-2}\na u pu\na \phi\phi^{p-1}=\int\beta(u) (\partial_z u \phi^p+pu\phi^{p-1}\partial_z\phi)-\int f u\phi^p.
\end{eqnarray}
Rearranging the order of integrals and after applying the H\"older inequality we get 
\begin{eqnarray}
\int|\na u|^p\phi^p&\leq&\int |\na u|^{p-1} p|u||\na \phi|\phi^{p-1}+\int\beta(u) (|\partial_z u| \phi^p+p|u|\phi^{p-1}|\partial_z\phi|)+ \int |f| |u|\phi^p\\\nonumber
&\le& \frac{\e^{p'}}{p'}\int|\na u|^p\phi^p+\frac1{\e^pp}p^p\int|u|^p|\na \phi|^p+\\\nonumber
&&+\beta(\sup_{B_1}|u|)\left[\frac{\e^p}{p}\int|\na u|^p\phi^p+ \frac1{\e^{p'}p'}\int\phi^{p}+p\int|u|\phi^{p-1}|\na \phi| \right]+ \int |f| |u|\phi^p .
\end{eqnarray}
Here $p'$ is the conjugate of $p$.
Choosing $\e$ small enough we infer the inequality
\begin{equation}
\int|\na u|^p\phi^p\le  \gamma \left[\int|u|^p|\na \phi|^p+\int\phi^{p}+\int|u|\phi^{p-1}|\na \phi| \right]+ \int |f| |u|\phi^p.
\end{equation}
where $\gamma$ depends on $\sup\limits_{B_1}|u|, \sup_{\C_L}|f|, p, a, \ell$ and $N$.
Now the required estimate follows from the fact that $\phi=1$ in $B_{\frac34}$.
\end{proof}

\subsection{Hopf's lemma}
\begin{lem}
Let $u$ be a weak solution of $\Delta_pu-a\partial_zu=0$ in  $B_{r}$, $u\in C(\overline{B_r})$ such that $u\ge 0$ in $B_r$ and $u(X_0)=0$ for some
$X_0\in \partial B_r$. Then 
$$\frac{\partial u(X_0)}{\partial \nu}<0$$
where $\nu$ is the unit outer normal at $X_0$. If the normal derivative does not exist then 
\[\limsup_{X\to X_0}\frac{u(X_0)-u(X)}{|X_0-X|}<0.\]
\end{lem}

\begin{proof}

Let $b=\gamma(e^{-\lambda|X|^2}-e^{-\lambda r^2})$
then (see for instance \cite{Aram JDE})
$$\Delta_p b=\gamma\lambda e^{-\lambda |X|^2}\left(2\gamma\lambda e^{-\lambda|X|^2}|X|\right)^{p-2}\left[4\lambda(p-1)|X|^2-2(N+p-2)\right]$$
Hence we have that 
$$\Delta_pb-a\partial_zb =2\gamma\lambda e^{-\lambda |X|^2}\left\{\left(2\gamma\lambda e^{-\lambda|X|^2}|X|\right)^{p-2}\left[2\lambda(p-1)|X|^2-(N+p-2)\right]+az\right\}$$
if we choose $\lambda$ sufficiently large, say $\lambda \ge \frac{2(N+p-2)}{p-1}$ and 
$\gamma=\frac{\inf_{B_{r/2}}u}{e^{-\lambda r^2/4}-e^{-\lambda r^2}}$ then 
we infer that $\Delta_pb-a\partial_zb\ge 0$.
By comparison principle Theorem 3.5.1 \cite{Serrin-book}  we have $b\le u$ and consequently
\[0>\limsup_{X\to X_0}\frac{b(X_0)-b(X)}{|X_0-X|}\ge \limsup_{X\to X_0}\frac{u(X_0)-u(X)}{|X_0-X|}\] 
If the normal derivative exists then this becomes 
$$0>\frac{\partial b}{\partial \nu} \ge \frac{\partial u}{\partial \nu}.$$
\end{proof}

\section*{\footnotesize Acknowledgements}
\footnotesize{The author would like to thank the referee for a number of important comments and suggestions.  }

\end{document}